\documentclass{amsart}
\usepackage{amsfonts,amssymb,graphicx}

\newtheorem{thm}{Theorem}
\newtheorem{lem}{Lemma}

\newcommand{\R}{\mathbb{R}}
\newcommand{\N}{\mathbb{N}}
\newcommand{\Sym}{\mathbf{S}}

\newcommand{\Id}{\mathrm{Id}}

\begin{document}
\title{Enumerating ODE Equivalent Homogeneous Networks}
\author{A. J. Windsor}
\address{Department of Mathematical Sciences, University of Memphis,
  Memphis, TN 38152-3240, U.S.A. }
\email{awindsor@memphis.edu}

\begin{abstract}
  We give an alternative criterion for ODE equivalence in identical
  edge homogeneous coupled cell networks. This allows us to give a
  simple proof of Theorem 10.3 of Aquiar and Dias, which characterizes
  minimal identical edge homogeneous coupled cell networks. Using our
  criterion we give a formula for counting homogeneous coupled cell
  networks up to ODE equivalence. Our criterion is purely graph
  theoretic and makes no explicit use of linear algebra.
\end{abstract}

\maketitle

\section{Introduction}

Coupled cell networks are used to represent systems of coupled
dynamical systems schematically. Such systems appear either in various
biological systems. Networks of eight coupled cells modeling central
pattern generators in quadrupeds can be used to recover the primary
animal gaits \cite{Gaits2,Gaits1}. One of the important conclusions of
the theory of coupled cell systems is that the network itself imposes
constraints on the possible behaviors of the system even when we lack
detailed knowledge of the behavior of the cells within the network. A
recent application to head movement that illustrates the importance of
this is \cite{LieJune}.  For further applications see
\cite{StewartNature}.

Mathematically coupled cell networks are a subclass of vertex and edge
labeled directed multigraphs with loops.  Vertices with the same label
represent multiple copies of the same dynamical system. Edge labels
represent the type of coupling. A compatibility condition is imposed
that requires every vertex with a given label to receive the same set
of coupling types as inputs. As with any class of graphs there is a
natural notion of \emph{isomorphic} coupled cell networks induced by
bijections between the sets of vertices. This representation of
coupled cell networks follows \cite{Synchrony}. An alternative
approach is outlined in \cite{Field}.

Following Stewart and Golubitsky one may associate to each coupled
cell network a class of ordinary differential equations that are
compatible with the network structure, the class of coupled cell
systems associated to a coupled cell network. As was pointed out in
\cite{Synchrony} it is possible for
non-isomorphic coupled cell networks to have the same class of coupled
cell systems. In this case we term the two coupled cell networks
O.D.E. equivalent. We will give a full description of the coupled cell
systems associated to a coupled cell network for the simple case of
identical edge homogeneous coupled cell networks. For the definition
in the general case see \cite{DiasLinear}. 

Aguiar and Dias \cite{Minimal} examine the structure of O.D.E. equivalence
classes for such coupled cell networks. They find a collection of
\emph{canonical normal forms} - a collection of networks whose number
of edges is minimal within the equivalence class. This they term the
\emph{minimal} subclass. 

In this paper we consider the simplest type of coupled cell networks,
the identical edge homogeneous coupled cell networks. These are simply
directed multigraphs with loops where every vertex has the same
indegree. Aldosray and Stewart gave an enumeration of these networks
\cite{StewartCounting} counted up to isomorphism.

Using a simpler method specific to the case of homogeneous networks we
recover Theorem 10.3 of \cite{Minimal} which characterizes the minimal
subclass in this case. Furthermore, we are able to give a recursive
formula for enumerating the minimal systems with a given number of
vertices and edges.

\section{Coupled Cell Systems, Coupled Cell Networks, and
  O.D.E. Equivalence.} 

We will deal exclusively with identical edge homogeneous coupled cell
networks, hereafter referred to simply as networks.  

Mathematically such a network is a directed multigraph where loops are
allowed and where every vertex has the same in-degree. If the constant
in-degree is $r$ we will call the network degree $r$.  A directed
multigraph consists of a set of vertices $V$ and a multiset of edges
$E$ with elements in $V \times V$. A multiset may be thought of as a
function $E: V \times V \rightarrow \N$; we call this function the
\emph{edge multiplicity function}. The condition that every vertex has
the same in-degree, $r$, is then $v \in V$ $\sum_{u \in V} E(u,v) =
r$.

Given an $n$ cell degree $r$ network $G=(V,E)$, a choice of finite
dimensional phase space $P=\R^d$, and a function $F: P \times P^r
\rightarrow P$ such that $F(x_1; y_1,\dots y_r)$ is invariant under
all permutations of the variables $y_1, \cdots, y_d$, we may produce a
vector field on $P^n$. The vector field for the variable $x_i$
associated to cell $i$ is
\begin{equation*}
  \dot{x}_i = F(x_i; x_{j^{(i)}_1},\dots x_{j^{(i)}_r})
\end{equation*}
where $j^{(i)}_1, \dots, j^{(i)}_r$ are the source cells for the $r$
arcs that terminate at vertex $i$.  The complete system is
\begin{eqnarray}
    \dot{x}_1 &=& F(x_1; x_{j^{(1)}_1},\dots x_{j^{(1)}_{r\phantom{1}}}) \nonumber\\
  &\vdots& \nonumber\\
  \dot{x}_n &=& F(x_n; x_{j^{(n)}_1},\dots x_{j^{(n)}_{r\phantom{1}}}) \nonumber
\end{eqnarray}
The set of such vector fields is a subset of the vector fields on
$P^n$.  

A vector field obtained from a coupled cell network $G$ by a choice of
phase space and function $F$ is referred to as a \emph{coupled cell
  system}, or an \emph{admissible vector field}, associated to $G$.
We may consider the class of all admissible vector fields for a given
network $G$ and phase space $P$.  We will denote this class of vector
fields by $\mathfrak{X}_G^P$.

\textbf{Definition:} Two coupled cell networks $G_1$ and $G_2$ are
called \emph{O.D.E. equivalent} if there exists a network $G'_2$
isomorphic to $G_2$ such that for all choices of phase space $P$
\begin{equation*}
  \mathfrak{X}_{G_1}^P = \mathfrak{X}_{G_2'}^P.
\end{equation*}

More prosaically, an $n$ cell degree $r_1$ network $G_1$ and an $n$
cell degree $r_2$ network $G_2$ are called \emph{O.D.E. equivalent} if
there exists a network $G'_2$ isomorphic to $G_2$ such that
\begin{enumerate}
\item for all choices of phase space $P$ and function $F_1: P \times
  P^{r_1} \rightarrow P$ there exists a function $F_2: P \times
  P^{r_2} \rightarrow P$ such that for all vertices $i$
  \begin{equation*}
    F_1(x_i; x_{j^{(i)}_1},\dots x_{j^{(i)}_{r_1\phantom{1}}}) =  F_2(x_i; x_{k^{(i)}_1},\dots x_{k^{(i)}_{r_2\phantom{1}}})   
  \end{equation*}
  where $j^{(i)}_1, \dots, j^{(i)}_{r_1}$ are the source cells for the
  $r_1$ arcs that terminate at cell $i$ in network $G_1$ and
  $k^{(i)}_1, \dots, k^{(i)}_{r_2}$ are the source cells for the $r_2$
  arcs that terminate at cell $i$ in network $G_2'$.

\item for all choices of phase space $P$ and function $F_2: P \times
  P^{r_2} \rightarrow P$ there exists a function $F_1: P \times
  P^{r_1} \rightarrow P$ such that for all vertices $i$
  \begin{equation*}
    F_1(x_i; x_{j^{(i)}_1},\dots x_{j^{(i)}_{r_1\phantom{1}}}) =  F_2(x_i; x_{k^{(i)}_1},\dots x_{k^{(i)}_{r_2\phantom{1}}})   
  \end{equation*}
  where $j^{(i)}_1, \dots, j^{(i)}_{r_1}$ are the source vertices for
the $r_1$ arcs that terminate at vertex $i$ in network $G_1$ and
$k^{(i)}_1, \dots, k^{(i)}_{r_2}$ are the source vertices for the $r_2$
arcs that terminate at vertex $i$ in network $G_2'$.

\end{enumerate}

If we  consider $P=\R$ and linear
functions $F_1$ and $F_2$  then we obtain the notion of \emph{linear
  equivalence}. It is shown in \cite{DiasLinear} that linear equivalence and
O.D.E. equivalence are equivalent.

\section{Network Operations that Preserve O.D.E. equivalence}
\label{sec:netw-oper-that}

In this section we introduce two operations that can be performed on a
network that preserve the O.D.E. equivalence class. Since we are
dealing exclusively with homogeneous networks these operations are a
small part of the network operations considered in
\cite{Minimal}. Both Lemma \ref{lem:Operations} and Lemma
\ref{lem:Reduced} can be deduced from the more general arguments in
\cite{Minimal}, in particular from Proposition 7.4. For completeness
we give proofs of both Lemma \ref{lem:Operations} and Lemma
\ref{lem:Reduced} using only what is required for our simpler
case. That we may consider only these two network operations and not
more general operations is crucial for the results in Section \ref{sec:Enumerate}.

Here we give go two simple operations on networks that preserve the
O.D.E. equivalence class of the network. 
\begin{enumerate}
\item Adding loops: A single loop is added to all vertices
  in the network. 
\item $k$-Splitting edges: Each edge in the network is replaced by $k$
  identical copies of the edge. 
\end{enumerate}
Intuitively, it should be clear that these operations preserve the
O.D.E. equivalence class of the network; however, a formal proof is
surprisingly difficult if one does not use the notion of linear
equivalence. 

\begin{lem}\label{lem:Operations}
  If network $G'$ is obtained from network $G$ by
  either of the two network operations above then $G$ and
  $G'$ are O.D.E. equivalent. 
\end{lem}

\begin{proof}
  Using \cite{DiasLinear} it is enough to prove that the two networks
  are equivalent when the variables $x_i$ are taken to be in $\R$ and
  the function $F$ is taken to be linear. In this case we observe that
  for a degree $r$ network the function $F$ must take the form
  \begin{equation*}
    F(x; y_1, \dots, y_r) = a \, x + b (y_1 + \dots + y_r). 
  \end{equation*}

  Consider a degree $r$ network. Adding a loop to every vertex we
  obtain a degree $r+1$ network.

  Given a function $F_r: \R^{r}\rightarrow \R$ defined by
  \begin{equation*}
    F_r(x;y_1, \dots,y_r) = a \, x + b (y_1 + \dots + y_r). 
  \end{equation*}
  we define a function $F_{r+1}: \R^{r+1} \rightarrow \R$ by
  \begin{equation*}
    F_{r+1}(x; y_1, \dots, y_{r+1})= (a-b) \, x + b( y_1 + \dots +y_{r+1}) . 
  \end{equation*}
  Clearly we have $F_{r+1}(x;x, y_1, \dots, y_r) = F_r(x;y_1, \dots,
  y_r)$ and consequently the linear vector fields admissible for the
  degree $r$ network are a subset of the linear vector fields
  admissible for the $r+1$ degree network. We can easily go the other
  direction.  Given any function $F_{r+1}: \R^{r+1} \rightarrow \R$ of
  the form
  \begin{equation*}
    F_{r+1}(x; y_1, \dots, y_{r+1})= a \, x + b( y_1 + \dots +y_{r+1})
  \end{equation*}
  we may define a function $F_{r}: \R^{r} \rightarrow \R$ by
  \begin{equation*}
    F_{r}(x; y_1, \dots, y_{r})= (a+b) \, x + b( y_1 + \dots +y_r) . 
  \end{equation*}
  Again we have $F_r(x;y_1, \dots, y_r) = F_{r+1} (x; x, y_1, \cdots ,
  y_r)$ and consequently we see that the two networks have precisely
  the same set of admissible linear vector fields.

  Consider a degree $r$ network. Performing the edge splitting
  operation we obtain a degree $k \times r$ network. Given any
  function $F_r:R^{r}\rightarrow \R$ of the form
  \begin{equation*}
    F_r(x;y_1, \dots,y_r) = a \, x + b (y_1 + \dots + y_r). 
  \end{equation*}
  we may define a function $F_{k \times r}: \R^{k \times r}
  \rightarrow \R$ by
  \begin{equation*}
    F_{k \times r}(x; y_1, \dots, y_{k\times r})= a \, x + \frac{b}{k} (
    y_1 + \dots + y_{k \times r}) . 
  \end{equation*}
  Clearly we have
  \begin{equation}
    F_{k \times r}(x;\overbrace{y_1,\dots,
      y_1}^{k\mathrm{-times}},  \dots, \overbrace{y_r,\dots,
      y_r}^{k\mathrm{-times}}) = F_r(x;y_1, \dots,
    y_r)\label{eq:Fkr}  
  \end{equation}
  and consequently the linear vector fields admissible for the degree
  $r$ network are a subset of the linear vector fields admissible for
  the degree $k\cdot r$ network.

  Given any function $F_{k \times r}: \R^{k \times r} \rightarrow \R$
  of the form
  \begin{equation*}
    F_{k \times r}(x; y_1, \dots, y_{k\times r})= a \, x + b (
    y_1 + \dots + y_{k \times r}) 
  \end{equation*}
  we may define a function $F_{r}: \R^{r} \rightarrow \R$ by
  \begin{equation*}
    F_{r}(x; y_1, \dots, y_{r})= a\, x + k \, b( y_1 + \dots +y_r) . 
  \end{equation*}
  Again equation \eqref{eq:Fkr} holds, and consequently we see that the
  two networks have precisely the same set of admissible linear vector
  fields.

  In both cases we see that the operation produces a new network with
  precisely the same set of admissible linear vector fields. Thus we
  have that the operations preserve the O.D.E. equivalence class.
\end{proof}

The operations create a network with a larger degree. However, when a
network has the required structure, the inverse of these operations
may be applied to produce a network with a smaller degree. 

  \begin{lem}\label{lem:Reduced}
    For any identical edge homogeneous coupled cell network $G$, there
    exists an O.D.E. equivalent network $G_M$ with the following
    properties:
    \begin{enumerate}
    \item At least one vertex has no loops, and 
    \item The greatest common divisor of the multiplicities of the
      edges is 1.
    \end{enumerate}
    We will refer to $G_M$ as a \emph{reduced} network
    associated to $G$. If $G$ is not a reduced network then $G_M$ has
    a lower degree than $G$.
  \end{lem}

  \begin{proof}
    Let $s$ denote the minimum number of loops on a vertex in
    $G$. Consider the new network $G'$ formed by removing exactly $s$
    loops from every vertex. Clearly $G'$ has a vertex with no
    loops. Since $G$ may be obtained from $G'$ by adding $s$ loops we
    see that $G$ and $G'$ are O.D.E. equivalent. Let $d$ denote the
    greatest common divisor of the edge multiplicities in $G'$. We may
    form a new network $G_M$ by dividing all the edge multiplicities
    by $d$. Since $G'$ had a vertex with no loops so does $G_M$. The
    greatest common divisor of the edge multiplicities of $G_M$ is 1
    by construction. Since we may obtain $G'$ from $G_M$ by splitting
    each edge into $d$ edges we see that $G'$ and $G_M$ are
    O.D.E. equivalent by Lemma \ref{lem:Operations}. Thus $G$ and
    $G_M$ are O.D.E. equivalent and $G_M$ has the required
    properties. 
  \end{proof}

    The use of $G_M$ to denote the reduced network is not
    accidental. We will now show that $G_M$ is indeed the unique
    minimal network in the O.D.E. equivalence class of $G$.  Since any
    network is O.D.E. equivalent to such a reduced network, it
    suffices to show that two reduced networks that are
    O.D.E. equivalent are isomorphic.

  \begin{lem}
    If $G_1$ and $G_2$ are reduced network,s and $G_1$ and $G_2$ are
    O.D.E. equivalent, then $G_1$ and $G_2$ are isomorphic.
  \end{lem}

  \begin{proof} Let $G_2'$ be the network isomorphic to $G_2$ which
  appears in the definition of O.D.E. equivalence. We will show that
  $G_1$ and $G_2'$ are equal. If we take the phase space $P$ for the
  cells to be $\R$ and consider linear functions of the form $F(x,
  y_1, \dots, y_r) = a \, x + b (y_1 + \dots + y_r)$, then we see that
  for any choice of $a_1, b_1$ there must exist $a_2, b_2$, and for
  any choice of $a_2, b_2$ there must exist $a_1, b_1$, such that
  \begin{equation*}
    (a_1 \Id + b_1 A) \, x= (a_2 \Id + b_2 B) \, x
  \end{equation*}
  where $A$ is the adjacency matrix associated to $G_1$, $B$ is the
  adjacency matrix associated to $G_2'$, and $x = ( x_1, \dots, x_n)^t
  \in \R^n$.  Since this holds for all $x \in \R^n$ we must have
  \begin{equation}\label{eq:Matrix}
    a_1 \Id + b_1 A= a_2 \Id + b_2 B.
  \end{equation}
  This matrix condition can be reduced to a system of linear equations
  of two types:
  \begin{eqnarray}
    a_1 + b_1 A_{ii} &=a_2 + b_2 B_{ii} & \quad 1 \leq i \leq n\label{eq:Diagonal} \\
    \qquad b_1 A_{ij} &= b_2 B_{ij}&  \quad 1\leq i,j \leq n, i\neq j \label{eq:OffDiagonal}
  \end{eqnarray}
  Since both $A$ and $B$ have a zero on the diagonal they must have
  some non-zero off diagonal entries in order to have the required row
  sums. Now by \eqref{eq:OffDiagonal} we see that $b_1$ and $b_2$ must
  have the same sign and that $A_{ij} \neq 0$ if and only if $B_{i j}
  \neq 0$.

  Since both $A$ and $B$ have at least one zero entry on the diagonal,
  either there must be an $1 \leq i \leq n$ such that $A_{ii} =
  B_{ii}=0$ or there must exist $i\neq j$ such that $A_{ii}=0$ but
  $B_{ii} >0$ and $A_{jj} >0$ but $B_{jj}=0$. If we assume that there
  exists $1 \leq i ,j \leq n$ with $i\neq j$ such that $A_{ii}=0$ but
  $B_{ii} >0$ and $A_{jj} >0$ but $B_{jj}=0$, then we obtain
  \begin{eqnarray}
    \label{eq:1}
    \qquad \quad a_1 &=& a_2 + b_2 B_{ii}\\
    a_1 + b_1 A_{jj} &=& a_2
  \end{eqnarray}
  from which we immediately get $-b_1 A_{jj} = b_2 B_{ii}$ which
  contradicts our earlier observation that $b_1$ and $b_2$ must have
  the same sign. Thus there exists an $1 \leq i \leq n$ such that
  $A_{ii} = B_{ii}=0$ and we can obtain from \eqref{eq:Diagonal} that
  $a_1 = a_2$.

  Thus we must have
  \begin{equation*}
    b_1 A_{i j} = b_2 B_{ij}
  \end{equation*}
  for all $1 \leq i,j \leq n$. Now $b_2$ divides $b_1 A_{ij}$ for all
  $i,j$. Since the greatest common divisor of the entries of $A$ is 1
  we must have $b_2$ divides $b_1$. Similarly $b_1$ divides $b_2
  B_{ij}$ for all $i,j$. Since the greatest common divisor of the
  entries of $B$ is 1, we must have $b_1$ divides $b_2$. Since $b_1$
  and $b_2$ have the same sign, we must have $b_1 = b_2$.

  Finally we are able to conclude that $A = B$ so $G_1$ is equal to
  $G_2'$ as claimed.
\end{proof}

\section{Examples}

First we show how Figure 1 and Figure 2 of \cite{Minimal} are related
using our network operations. 
\begin{figure}[h]
  \centering
  \begin{tabular}{ccc}
    \includegraphics[scale=.3]{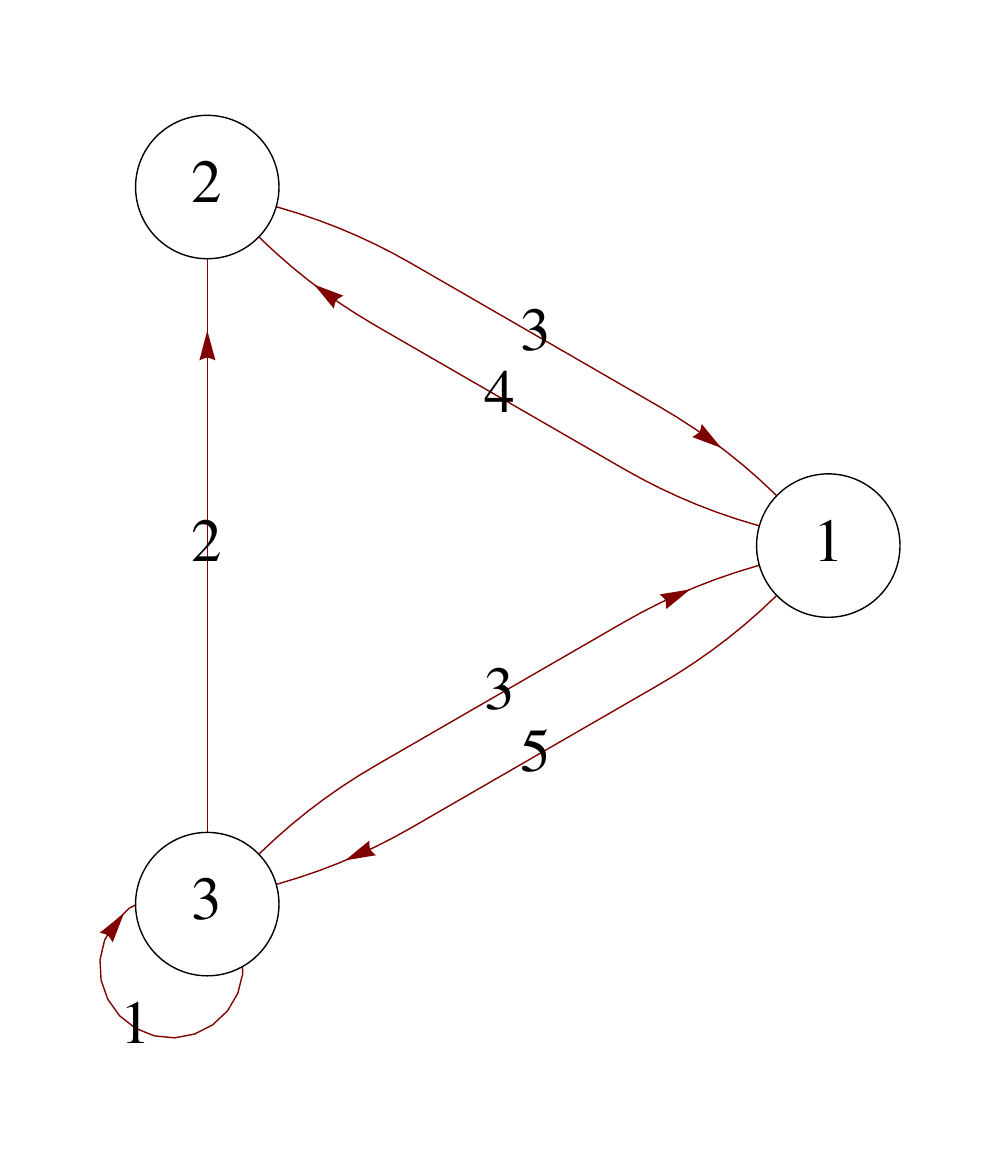}
    & \includegraphics[scale=.3]{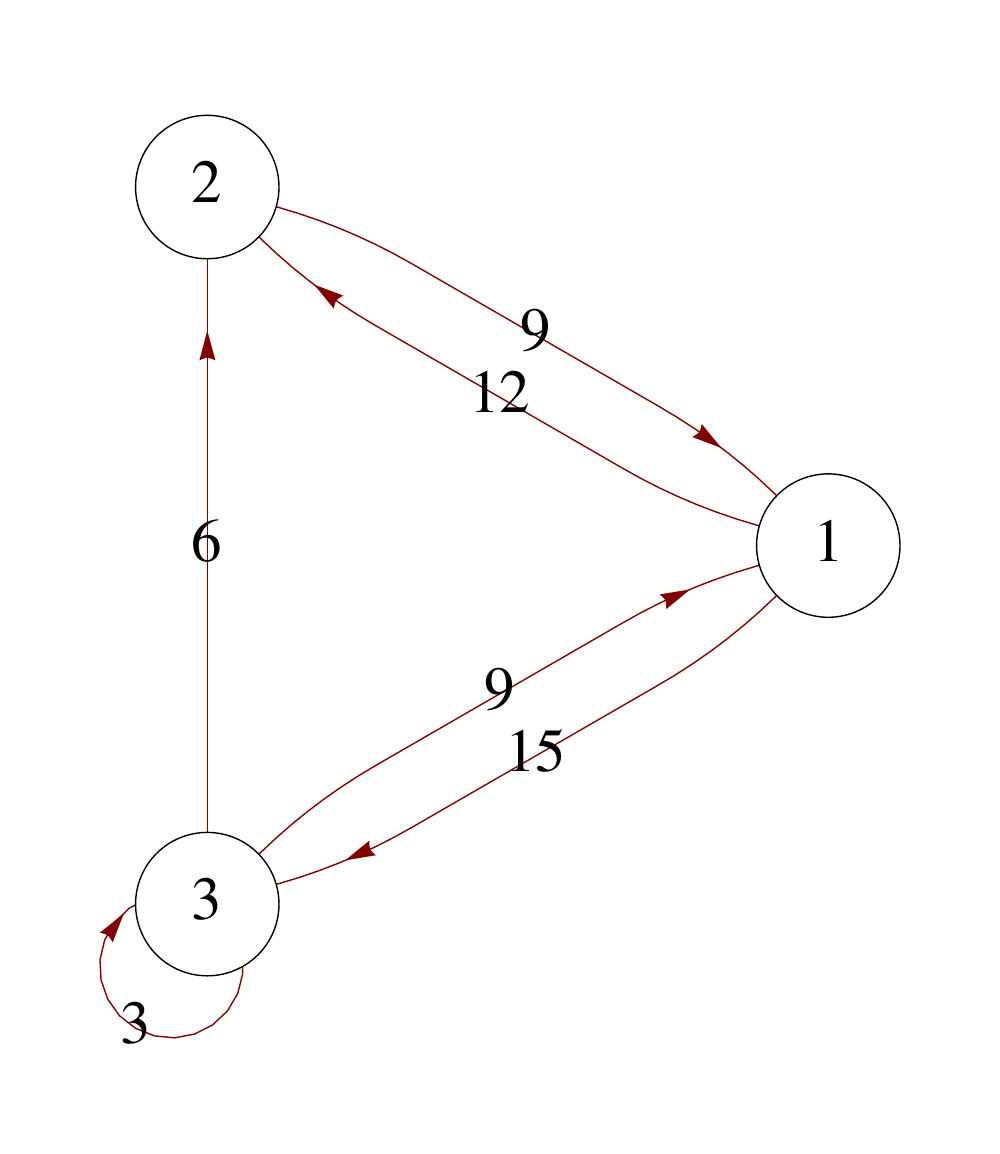}& 
    \includegraphics[scale=.3]{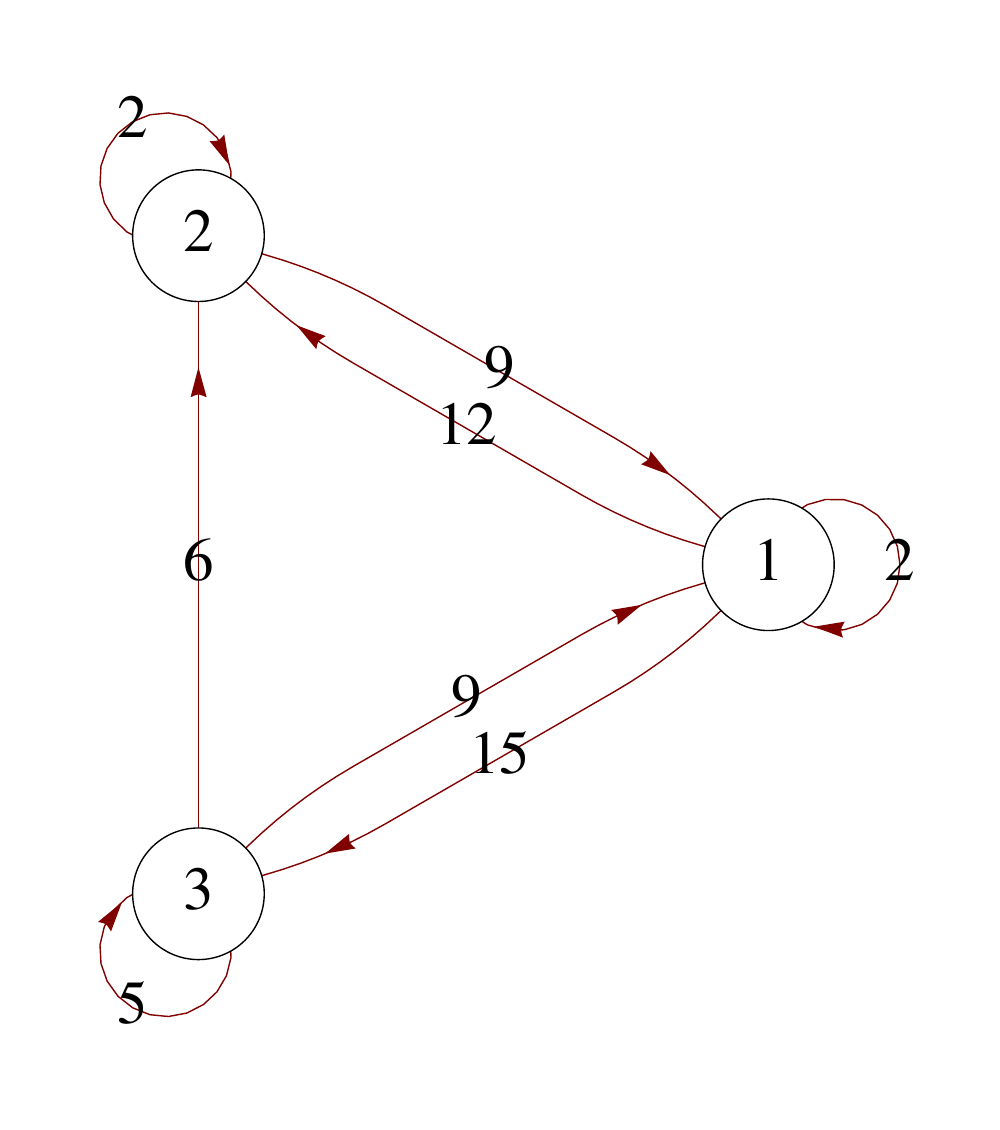}\\ 
    (1)&(2)&(3)
  \end{tabular}
 \caption{Transforming Figure 1 to Figure 2 of \protect \cite{Minimal} using
  network operations. Edge labels represent edge multiplicities. }
  \label{fig:Transforming}
\end{figure}

Referring to our Figure \ref{fig:Transforming} notice that network (1)
satisfies our criterion for being a minimal network. If we split each
edge of network (1) into 3 edges then we obtain network (2), which is
O.D.E. equivalent to network (1). If we now adjoin 2 loops to each
vertex of network (2), then we obtain network (3), which is O.D.E. to
network (2) and hence O.D.E. equivalent to network (1).

Next we apply the results of the previous section to the connected 3
cell degree 2 networks examined in \cite{Leite}. They note
that up to permutation there are 38 connected 3 cell degree 2 networks
but that 8 of them are O.D.E. equivalent to the lower degree
networks. Each of these 8 is obtained from one of the 4 minimal
connected 3 cell degree 1 networks by either adjoining a loop to
every cell or by doubling all the edges, see Figure \ref{fig:Example}.

\begin{figure}[h]
  \centering
  \includegraphics[scale = .7]{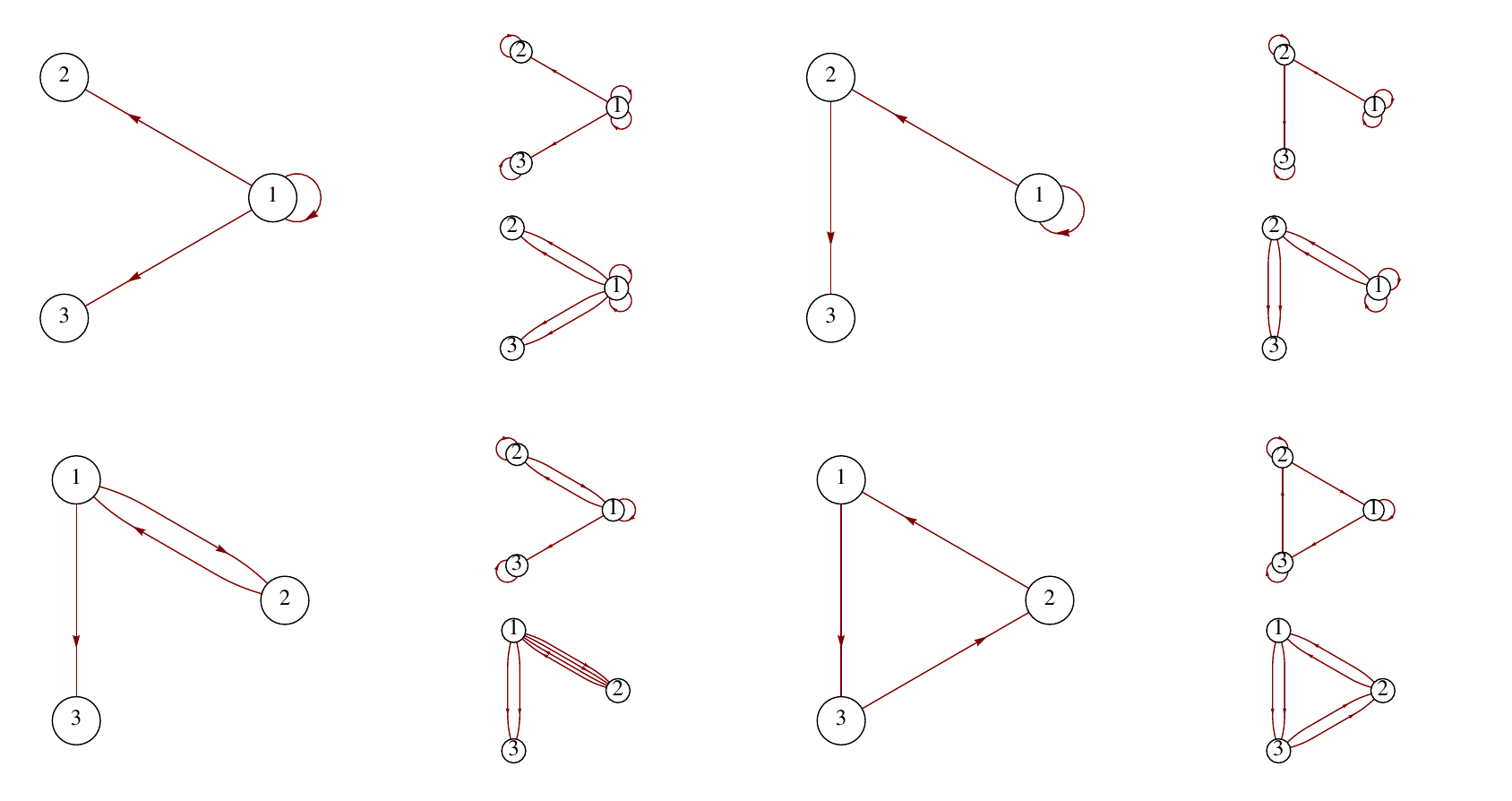}
  \caption{The minimal connected 3 cell degree 1 networks and their
    associated connected 3 cell degree 2 networks}
  \label{fig:Example}
\end{figure}

\section{Enumeration}
\label{sec:Enumerate}

We begin by outlining the work of Aldosray and Stewart in enumerating
homogeneous coupled cell networks. They use the counting result known
as Burnside's Lemma to enumerate all identical edge homogeneous
coupled cell networks with $n$ cells and degree $r$ counted up to
isomorphism. 

To be explicit let us take $V= \{1, \dots, n\}$. Let us denote the set
of all multigraphs on $V$ with constant in-degree $r$ by
$\Omega_{n,r}$. The group of bijections on $V$ is the symmetric group
on $n$ elements, denoted $\Sym_n$. Each such bijection induces a
map on $\Omega_{n,r}$. Thus we have a group action of $\Sym_n$
on $\Omega_{n,r}$.  Two networks are related by $\Sym_n$ if and
only if they are isomorphic networks.  Since we are counting the
networks up to isomorphism what we actually want to count is the
number of distinct $\Sym_n$ orbits in $\Omega_{n,r}$. Burnside's
Lemma is a tool for counting the number of orbits of a group action,
it states
\begin{equation*}
  |\mathrm{Orb}_{\Omega_{n,r}}(\Sym_n)| = \frac{1}{|\Sym_n|} \sum_{g
    \in \Sym_n} |\mathrm{Fix}_{\Omega_{n,r}} (g)|. 
\end{equation*}
where $\mathrm{Fix}_{\Omega_{n,r}} (g) = \{ \omega \in \Omega_{n,r} :
g \cdot \omega = \omega \}$.  If $g$ and $h$ are conjugate elements of
$\Sym_n$ then $|\mathrm{Fix}_{\Omega_{n,r}} (g)| =
|\mathrm{Fix}_{\Omega_{n,r}} (h)|$ and consequently we may sum over
conjugacy classes rather than individual elements of $\Sym_n$. Suppose
that $C_1 \, \dots, C_m$ are the conjugacy classes in $\Sym_n$. Let
$g_i$ be some representative of the conjugacy class $C_i$. We may
write our sum as 
\begin{equation}\label{eq:ModBurnside}
  |\mathrm{Orb}_{\Omega_{n,r}}(\Sym_n)| = \frac{1}{|\Sym_n|} \sum_{i
    =1}^m |C_i| |\mathrm{Fix}_{\Omega_{n,r}} (g_i)|. 
\end{equation}

There is a bijection between conjugacy classes of $\Sym_n$ and
partitions of the integer $n$. Following \cite{StewartCounting} we
will denote a partition of $n$
\begin{equation*}
  \alpha_1 \cdot 1 + \alpha_2 \cdot 2 +\cdots +\alpha_n \cdot n = n 
\end{equation*}
by $[1^{\alpha_1}2^{\alpha_2}\dots n^{\alpha_n}]$. The multiplicative
form of this notation is perhaps unfortunate but should not cause
confusion. The strength of this notation becomes apparent when we
agree that if $\alpha_i =0$ then the $i^{\alpha_i}$ term in the
expression may be omitted. Using this notation the 7 partitions of
$n=5$ may be expressed as follows:
\begin{equation*}
  \begin{array}{l@{\hspace{.5 cm}}l@{\hspace{1 cm}}l@{\hspace{.5 cm}}l}
    5 \cdot 1 & [1^5] & 1 \cdot 1 + 2 \cdot 2&[1^12^2]\\
    3 \cdot 1 + 1 \cdot 2&[1^32^1]&1 \cdot 2 + 1 \cdot 3&[2^13^1]\\
    2 \cdot 1 + 1 \cdot 3&[1^23^1]&    1 \cdot 5&[5^1] \\
    1 \cdot 1 + 1 \cdot 4&[1^14^1]
  \end{array}
\end{equation*}
The set of all partitions of $n$ will be denoted by $\Pi_n$. An
element of $\Sym_5$ can be associated to each $\rho \in \Pi_n$ as follows:
\begin{equation*}
  \begin{array}{l@{\hspace{.5 cm}}l@{\hspace{1 cm}}l@{\hspace{.5 cm}}l}
     {}[1^5] &(1)(2)(3)(4)(5)&[1^12^2] &(1)(2\,3)(4\,5)\\
    {} [1^32^1] &(1)(2)(3)(4\,5)&[2^13^1] &(1\,2)(3\,4\,5)\\
    {}[1^23^1]&(1)(2)(3\,4\,5)&[5^1]& (1\,2\,3\,4\,5)\\
    {}[1^14^1]&(1)(2\,3\,4\,5)
   \end{array}
 \end{equation*}
Every permutation in $\Sym_5$ is conjugate to one of the permutations
that correspond to a partition of 5. 

Every permutation $\sigma \in \Sym_n$ may be written as a product of
disjoint cycles in a fashion that is unique up to the order to the
cycles. The lengths of these cycles form a partition on $n$ called the
\emph{cycle type} of the permutate $\sigma$. The permutation
corresponding to a given cycle type is called the \emph{normal form}
of the cycle type. Every permutation is conjugate to the normal form
of its cycle type.

 Looking at the formula
\eqref{eq:ModBurnside} we see that it would be advantageous to know the
size of the conjugacy class associated to a given partition of
$n$. The size of the conjugacy class corresponding to
$[1^{\alpha_1}2^{\alpha_2}\cdots n^{\alpha_n}]$ is
\begin{equation}
  \label{eq:3}
  \frac{n!} {1^{\alpha_1}2^{\alpha_2}\cdots
n^{\alpha_n} \alpha_1! \alpha_2! \dots \alpha_n!}.
\end{equation}
If we consider the partition determines the pattern of parentheses 
\begin{equation*}
  \begin{array}{ll}
    [1^22^23^1] & (\rule{10 pt}{.5 pt} )(\rule{10 pt}{.5 pt} ) (\rule{10 pt}{.5 pt} \,
    \rule{10 pt}{.5 pt} ) (\rule{10 pt}{.5 pt}\,
    \rule{10 pt}{.5 pt} ) (\rule{10 pt}{.5 pt} \,\rule{10 pt}{.5 pt}\,
    \rule{10 pt}{.5 pt} )
  \end{array}
\end{equation*}
then $n!$ is the number of ways of writing $1, \dots, n$ in the
blanks.Observing that we can permute each cycle cyclically, that is
\begin{equation*}
  ( 1 2 3 ) \quad (2 3 1) \quad (3 1 2)
\end{equation*}
are all the same 3-cycle, we must factor out the 
$1^{\alpha_1}2^{\alpha_2}\cdots n^{\alpha_n} $ possible ways of
expressing all the cycles. Finally we
observe that we may permute cycles of the same length freely, so we
must factor out the  $a_1! \alpha_2 ! \dots \alpha_n!$ possible
orderings of the cycles.

The main difficulty in enumerating the orbits of $\Sym_n$  lies  in
determining the size of the fixed point set
$\mathrm{Fix}_{\Omega_{n,r}}(g_i)$. We will give the formula for this
here and refer the reader to the details in \cite{StewartCounting}. 

\textbf{Definition}: Given $\rho \in \Pi_n$ and $s \in \{1, \dots,
n\}$ we may define 
\begin{equation*}
  \Phi_{s, \rho} (z) = \prod_{k=1}^n (1- z^{\frac{k}{h}})^{-\alpha_k^\rho h}
\end{equation*}
where $h= \gcd (s,k)$.

Clearly $\Phi_{s,\rho}(z)$ is analytic about $0$ and hence we may
write
\begin{equation*}
  \Phi_{s, \rho} (z) = \sum_{r=1}^\infty \phi_r(s,\rho) z^r.
\end{equation*}

\begin{thm}[Theorem 8.3 \cite{StewartCounting}]
  Let $n, r\in \N\setminus \{0\}$ . Let $H_{n,r}$ denote the number of
  $n$ cell degree $r$ networks counted up to isomorphism.  $H_{n,r}$
  is given by
  \begin{equation*}
    H_{n,r} = \frac{1}{n!} \sum_{\rho \in \Pi_n}   \frac{n!} {1^{\alpha_1}2^{\alpha_2}\cdots
      n^{\alpha_n} \alpha_1! \alpha_2! \dots \alpha_n!} \prod_{k=1}^n
    \phi_r(k,\rho)^{\alpha_k^\rho} .
  \end{equation*}
\end{thm}

We use this theorem to generate Table \ref{tab:Hnr}. 
\begin{table}[h]
 \centering
   \begin{tabular}{rr|rrrrrr}
     &&\multicolumn{6}{c}{$r$}\\
     && 1 & 2 & 3 & 4 & 5 & 6 \\ \hline
     &1 & 1 & 1 & 1 & 1 & 1 & 1 \\
     &2 & 3 & 6 & 10 & 15 & 21 & 28 \\
     $n$ &3 & 7 & 44 & 180 & 590 & 1582 & 3724 \\
     &4 & 19 & 475 & 6915 & 63420 & 412230 & 2080827 \\
     &5 & 47 & 6874 & 444722 & 14072268 & 265076184 & 3405665412 \\
     &6 & 130 & 126750 & 43242604 & 5569677210 & 355906501686 & 13508534834704
\end{tabular}
\caption{The number of $n$ cell degree $r$ networks counted up to
  isomorphism, $H_{n,r}$. }
  \label{tab:Hnr}
\end{table}

This count however includes disconnected coupled cell networks. From
the perspective of dynamical systems we are interested only in the
connected identical edge coupled cell networks. A disconnected system
can be decomposed into a number of connected systems. Thus a
disconnected $n$ cell network corresponds to a partition of $n$ with
$\alpha_n = 0$ i.e. any partition of $n$ except $[n^1]$.

If we denote the number of connected $n$ cell degree $r$ networks by
$K_{n,r}$ then we may enumerate the number of disconnected coupled
cell networks as follows
\begin{equation*}
  \sum_{\stackrel{\rho \in \Pi_n}{ \alpha_n^\rho = 0}} \prod_{m=1}^{ n-1}
  {K_{m,r} +\alpha_m^\rho -1 \choose \alpha_m^\rho}
\end{equation*}
where 
\begin{equation*}
   {K_{m,r} +\alpha_m^\rho -1 \choose \alpha_m^\rho} 
\end{equation*}
is the number of ways of choosing $\alpha_m^\rho$ networks from the
$K_{m,r}$ distinct connected $m$ cell networks with replacement and
where order does not matter. From this we obtain

\begin{thm}[Theorem 10.1 \cite{StewartCounting}]
  Let $n, r\in \N\setminus \{0\}$ .  Let $K_{n,r}$ denote the number
  of minimal connected $n$ cell degree $r$ networks. We have $K_{1,r}
  = H_{1,r} =1$ and for $n \geq 2$
  \begin{equation*}
    K_{n,r} = H_{n,r} - \sum_{\stackrel{\rho \in \Pi_n}{\alpha_n^\pi =
        0} } \prod_{m=1}^{ n-1}
    {K_{m,r} +\alpha_m^\rho -1 \choose \alpha_m^\rho}.
  \end{equation*}
\end{thm}

We use this theorem to generate Table \ref{tab:Knr}.

\begin{table}[h]
  \centering
  \begin{tabular}{rr|rrrrrr}
&&\multicolumn{6}{c}{r}\\
  && 1 & 2 & 3 & 4 & 5 & 6 \\\hline
 &1 & 1 & 1 & 1 & 1 & 1 & 1 \\
 &2 & 2 & 5 & 9 & 14 & 20 & 27 \\
 n&3 & 4 & 38 & 170 & 575 & 1561 & 3696 \\
 &4 & 9 & 416 & 6690 & 62725 & 410438 & 2076725 \\
 &5 & 20 & 6209 & 436277 & 14000798 & 264632734 & 3403484793 \\
 &6 & 51 & 117020 & 42722972 & 5554560632 & 355631996061 & 13505066262007
\end{tabular}
  \caption{The number of connected $n$ cell degree $r$ networks
    counted up to isomorphism, $K_{n,r}$}
  \label{tab:Knr}
\end{table}

Now we will use the work of Section 3 to give a recursive formula for
enumerating the connected minimal coupled $n$ cell degree $r$
networks. 

\begin{thm}
  Let $M_{n,r}$ denote the number of minimal connected $n$ cell degree
  $r$ networks. For $n \geq 2$ we have $M_{n,1}= K_{n,1}$ and
  \begin{equation*}
    M_{n,r} = K_{n,r} - \sum_{s=1}^{r-1} \biggl\lfloor
    \frac{r}{s} \biggr\rfloor M_{n,s}.
  \end{equation*}
  For $n=1$ note that $M_{n,r} = 0$. 
\end{thm}

\textbf{Proof:} If a connected $n$ cell degree $r$ network is not
minimal then it is O.D.E. equivalent to a minimal $n$ cell $s$ degree
network where $s < n$. Given a minimal $n$ cell degree $s$ network $G$
the question thus becomes how many non-isomorphic $n$ cell degree $r$
networks can be obtained that are O.D.E. equivalent to $G$. We have
seen that any network $G'$ O.D.E. equivalent to a minimal network $G$
may be obtained from $G$ by a combination of adjoining loops and
splitting edges (and an isomorphism which we may ignore).  Let $A$ be
the operation of adjoining a root and $T_k$ the operation of
$k$-splitting the edge. Clearly we have $T_k \circ T_l = T_{k
  l}$. There is a commutation relation between $T_k$ and $A$, $T_k
\circ A = A^k \circ T_k$. Using this commutation relation we see that
any combination of adjoining loops and edge splitting can be reduced
to a single $k$-splitting for some $k \geq 1$ followed by adjoining
some number of loops. Given that $G$ has degree $s$ and $G'$ has
degree $r$ the possible choices of $k$ are constrained by $k s \leq
r$.  Thus there are $\lfloor r/s \rfloor$ possible values of
$k$. We then adjoin sufficiently many loops to bring the degree to
$r$.

The number of connected minimal $n$ cell degree $r$ networks is thus
given by
\begin{equation*}
  M_{n,r} = K_{n,r} - \sum_{s=1}^{r-1} \biggl\lfloor
\frac{r}{s} \biggr\rfloor M_{n,s}
\end{equation*}
with the initial condition that $M_{n,1}= K_{n,1}$ for $n\geq2$.  \qed

Using this theorem we generate Table \ref{tab:Mnr}.

\begin{table}[h]
  \centering
 \begin{tabular}{rr|rrrrrr}
&&\multicolumn{6}{c}{$r$}\\
  && 1 & 2 & 3 & 4 & 5 & 6 \\ \hline
 &1 & 0 & 0 & 0 & 0 & 0 & 0 \\
 &2 & 2 & 1 & 2 & 2 & 4 & 2 \\
 $n$&3 & 4 & 30 & 128 & 371 & 982 & 1973 \\
 &4 & 9 & 398 & 6265 & 55628 & 347704 & 1659615 \\
 &5 & 20 & 6169 & 430048 & 13558332 & 250631916 & 3138415822 \\
 &6 & 51 & 116918 & 42605901 & 5511720691 & 350077435378 & 13149391543076
\end{tabular}

  \caption{The number of minimal connected $n$ cell degree $r$ networks counted up to isomorphism, $M_{n,r}$. }
  \label{tab:Mnr}
\end{table}

It is interesting to note that the number of connected minimal $2$ cell degree
$r$ networks for $r\geq 2$ is given by $\phi(r)$ where $\phi$ is the
Euler totient function.  The appearance of the Euler Totient is
explained by the following network diagram:
\begin{figure}[h]
  \centering
  \includegraphics[scale = .7]{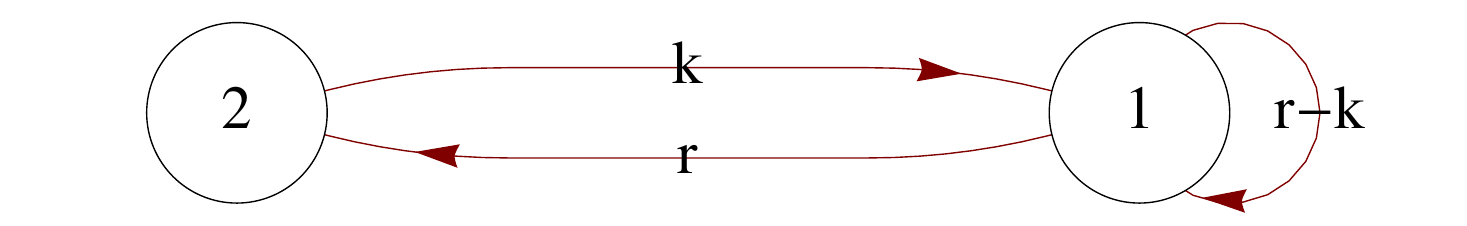}

  \caption{ A connected 2 cell degree $r$ network. Edge labels
    represent edge multiplicities.}
  
\end{figure}

In order for a 2 cell network to be minimal at least one vertex must
have no loops. Without losing generality we may suppose that vertex 2
has no loops. Thus vertex 2 must receive $r$ inputs from vertex 1. If
we let $k$, with $k \leq r$, denote the number of edges from vertex 2
to vertex 1 then we see that vertex 1 must have $r-k$ loops. If this
network is to be minimal then the three edge multiplicities, $r$, $k$,
and $r-k$, must be relatively prime. This occurs if and only if $r$
and $k$ are relatively prime. For a fixed $r$ the number of $1 \leq k
\leq r$ for which $r$ and $k$ are relatively prime is
$\phi(r)$. Provided that $r \geq 2$ we may exclude $k=0$ since then
$r-k=r$ and all edge multiplicities have divisor $r$ and hence the
network is not minimal. 

If $r=1$ then there are in fact two minimal 2 cell degree 1
networks. 
\begin{figure}[h]
  \centering
  \includegraphics[trim = 0in 0.8in 0in 0.8in]{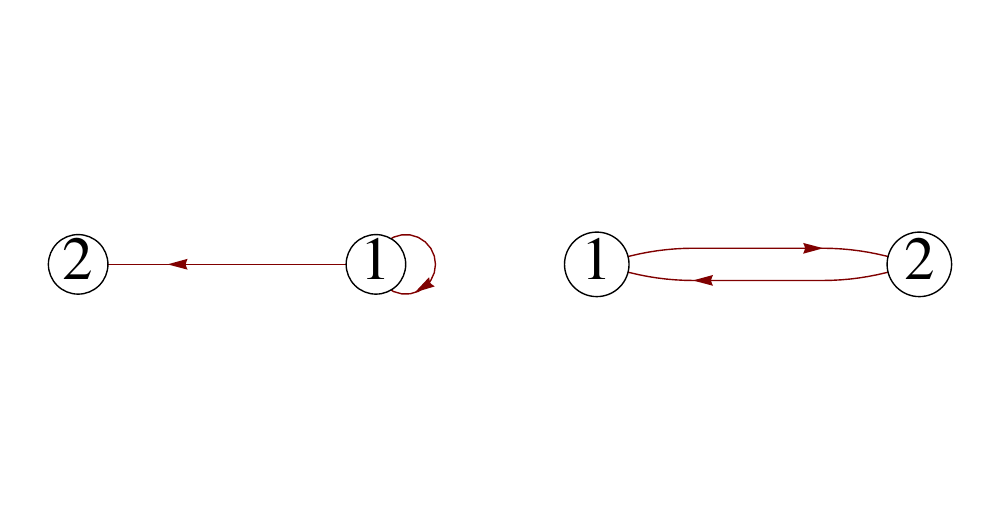}
  \caption{The two minimal 2 cell degree 1 networks. }
  \label{h}
\end{figure}

\bibliographystyle{plain}

\bibliography{Enumeration}

\end{document}